\documentclass[journal]{IEEEtran}
\ifCLASSINFOpdf
\else
   \usepackage[dvips]{graphicx}
\fi
\usepackage{url}

\usepackage{graphicx}
\usepackage{amsmath}
\usepackage{amssymb}
\usepackage{amsthm}
\usepackage{soul}
\setstcolor{blue}
\usepackage{enumerate}
\usepackage{color}
\usepackage{algorithm}
\usepackage{algorithmic}
\usepackage{pseudocode}
\usepackage{multirow}
\usepackage{booktabs}
\usepackage{subfig}
\usepackage{amsmath,amssymb,amsthm}
\usepackage{epstopdf}
\usepackage{ulem}
\usepackage{hyperref}

\usepackage{graphicx}
\usepackage{epsfig}
\newtheorem{theorem}{Theorem}[section]

\newtheorem{corollary}[theorem]{Corollary}
\begin{document}
\title{Nullspace Property for Optimality of Minimum Frame Angle Under Invertible Linear Operators}
\author{Pradip Sasmal, Prasad Theeda, Phanindra Jampana and Challa S Sastry
%\thanks{This paragraph of the first footnote will contain the date on which you submitted your paper for review. It will also contain support information, including sponsor and financial support acknowledgment. For example, ``This work was supported in part by the U.S. Department of Commerce under Grant BS123456.'' }
\thanks{Dept. of ECE, Indian Institute of Science, Bangalore, India. email:\{pradipsasmal\}@iisc.ac.in.}
\thanks{Vellore Institute of Technology (VIT) University, Vellore, Tamilnadu, India. email:\{prasad.theeda\}@vit.ac.in.}

\thanks{Indian Institute of Technology, Hyderabad, Telangana. 502285, India. Email:\{pjampana, csastry\}@iith.ac.in}

\thanks{This work has been submitted to the IEEE for possible publication. Copyright may be transferred without notice, after which this version may no longer be accessible.}
}

\markboth{Journal of \LaTeX\ Class Files, Vol. xx, No. xx, 2021}
{Shell \MakeLowercase{\textit{et al.}}: Bare Demo of IEEEtran.cls for IEEE Journals}
\maketitle

\begin{abstract}
Frames with a large minimum angle between any two distinct frame vectors are desirable in many present day applications. For a unit norm frame, the absolute value of the cosine of the minimum frame angle is also known as coherence. Two frames are equivalent if one can be obtained from the other via left action of  an invertible linear operator. Frame angles can change under the action of a linear operator. Most of the existing works solve different optimization problems to find an optimal linear operator that maximizes the minimal frame angle (in other words, minimizes the coherence). In the present work, nevertheless, we consider the question: Is it always possible to find an equivalent frame with smaller coherence for a given frame?. In this paper, we derive properties of the initial unit 
norm frame that can ensure an equivalent frame with strictly larger minimal frame angle compared to the initial one. It turns out that the nullspace property of a certain matrix obtained from the initial frame can guarantee such an equivalent frame. We also present the numerical results that support our theoretical claims. 
\end{abstract}


\IEEEpeerreviewmaketitle

\section{Introduction}
Frame theory has
%offers far and wide
applications in 
%diverse
fields such as signal processing, sparse representation theory and operator theory.  The coherence of a frame is defined as the largest absolute normalized inner
product between two distinct frame vectors. A finite frame can be represented as a matrix of full row rank.
%, which is an easily computable quantity.
%A finite frame with small coherence is said to be incoherent.
For a fixed number of elements, frames with the smallest coherence
%,  referred to as incoherent,  
are
called Grassmannian frames \cite{stro_2003}. Grassmanian frames attaining the Welch bound  are known as equiangular tight frames (ETFs) \cite{mixon_2012}. 
%ETFs attain minimal coherence and exist only for some particular dimensions. 
Incoherent frames play a significant role due to their ability in providing sparse representations. 

\par The field of sparse representation theory, popularly known as compressed sensing (CS) \cite{elad_2007, Kashin_2007}, recovers a sparse signal from a few of its 
%representation coefficients (often, referred to as 
linear measurements. Performance of several sparse recovery algorithms such as basis pursuit (BP) and orthogonal matching pursuit (OMP) depends 
on the coherence of the underlying frame.
Frames that satisfy the restricted isometry property (RIP)
\cite{Elad_2006}\cite{UoB} are known to allow for exact recovery of sparse
signals from a few of their linear measurements. However, in general, it
is computationally hard to verify the RIP of a
given frame. In contrast, the coherence of a frame, being easily computable, asserts RIP up to certain order \cite{Elad_2010}.
%{\color{red}{many CS citations can be removed. Plus add your J.Comp+Ram's Euler Sq. paper, Remove DeVore paper}}
%\sout{Compressed Sensing (CS) \cite{donoho_2003,elad_2007, cand_2006, cds_2001, Kashin_2007} aims at recovering a sparse signal from a few of its 
%representation coefficients (often, referred to as 
%linear measurements. Performance of several sparse recovery algorithms such as Basis Pursuit (BP) and Orthogonal Matching Pursuit (OMP) depends 
%on the coherence of the underlying frame.
%Frames that satisfy the Restricted Isometry Property (RIP)
%\cite{Donoho_2006, Elad_2006} are known to allow for exact recovery of sparse signals from a few of their linear measurements. However, in general, it is computationally hard to verify the RIP of a given frame. In contrast, the coherence of a frame (the largest absolute normalized inner product between two distinct columns) is an easily computable quantity that asserts RIP up to certain order \cite{Elad_2010}. Frames with small coherence are said to be incoherent. For a fixed number of elements, frames with the smallest coherence are called Grassmannian \cite{stro_2003}. Optimal Grassmanian frames are known as Equiangular Tight Frames (ETFs) \cite{mixon_2012}. ETFs attain minimal coherence and exist only for some particular dimensions. Incoherent frames play a significant role due to their ability in providing sparse representations. }

 Two linear systems of equations provided by two frames are equivalent if the underlying frames are equivalent. However, the sparse recovery properties of equivalent frames can be significantly different. Consequently, the performances of sparse recovery algorithms can be different. In \cite{zhang_2013}, the authors presented that the RIP of an equivalent frame can be bad compared to the initial frame, where as, in \cite{pra_2015}, the authors derived that the RIP constant of an equivalent frame can be improved. On the other hand, the coherence of two equivalent frames can also be different.

Several methods exist in the literature \cite{elad_2007}\cite{duarte_2009}\cite{li_2013}\cite{alemazkoor_2018}
for finding an equivalent frame with optimal coherence.
%was the end goal.
%to produce equivalent frames having smaller coherence compared to the initial frames.  
Although these  methods work well in practice, they do not possess theoretical guarantees for reduction in coherence.
In this work, however, we consider the question: `` For a given frame, is it possible to find an equivalent frame with smaller coherence?."  
%The answer to this question is \emph{no}. One class of such frames are Grassmanian frames as they already attain the minimal coherence. Consequently, the following obvious question arises ``what are properties of the initial frame that can guarantee an incoherent equivalent frame?." 
The main objective of the present work is to derive sufficient conditions on a frame that can ensure an equivalent frame possessing smaller coherence. We show that the existence of such an equivalent  frame
can be ascertained by checking for in-feasibility of a linear system of equations. The null space property of a certain matrix obtained from the initial frame ensures the existence of an equivalent frame with a strictly smaller coherence. 
The main contributions of the paper are summarized below:
\begin{itemize}
\item We derive the sufficient conditions that can guarantee existence of an equivalent frame having smaller coherence compared to the initial frame.
%We propose a convex optimization problem to obtain a preconditioner which results in low coherence.
%\item Deriving sufficient conditions for reduction in coherence with diagonal preconditioners using frame theory
%\item
%We derive sufficient conditions for the optimal preconditioner that can guarantee strict fall in coherence. 
%for general pre both in the noiseless and the noisy cases
%\item We show that coherence of an equivalent unit norm frame obtained by the left action of a diagonal invertible operator on a binary frame remains unchanged.
\item We present numerical results that validate our theoretical analysis.
\end{itemize}
%The paper is organized in six sections. 
%In section 2, we present the relevant basics of frame theory and compressed sensing. 
%In sections 2 and 3, we provide our main contributions and proofs. In section 4 and 5, we present numerical analysis and concluding remarks  respectively.  %In the next section, we provide basics of frame theory and compressed sensing.

\section{Basics of Frame theory}
\subsection{Frame Theory}
%Frames are overcomplete spanning systems, which are introduced as a generalization of bases~\cite{RM1, chris_2003, kova_2007, daube_1992, Gita_2013}. Frame theory has deep connections with Harmonic analysis, Operator theory, Coding theory \cite{holmes_2004, stro_2003} and Quantum theory \cite{eldar_2002}. Due to its widespread applicability, the area of frames has become an active area of research for many researchers from several fields. In signal processing, frames play a significant role in sparse signal recovery \cite{bruck_2009,Elad_2010}. 
A family of vectors $\{\phi_{i}\}^{M}_{i=1}$ in $\mathbb{R}^{m}$ is called a frame \cite{chris_2003} for $\mathbb{R}^{m}$, if there exist constants $0 < A \leq B < \infty$ such that
\begin{align*}
A\left\|z\right\|^{2} \leq \sum^{M}_{i=1}\left|\left\langle z, \phi_{i}\right\rangle\right|^{2} \leq B\left\|z\right\|^{2}, \forall z \in \mathbb{R}^{m},   
\end{align*}
\noindent where $A$ and $B$ are called the lower and upper frame bounds respectively. 
%The matrix $\Phi_{m \times M} = [ \phi_{1} \dots \phi_{M}]$ with $\phi_{i}$ as columns is known as the frame synthesis operator. The adjoint of synthesis operator $\Phi$ is known as analysis operator $\Phi^{*}=\Phi^{T}$. The frame operator $F_{\Phi}$ is given by $F_{\Phi}=\Phi \Phi^{T}.$  In this paper we refer to  $\Phi$ as a frame and also as a matrix depending on the context.  
%Frames are characterized depending on their properties as outlined below \cite{cas_2012}:
%\begin{itemize}
%\item 
If $A=B$, then $\{ \phi_{i} \}^{M}_{i=1}$ is an $A-$tight frame. 
%\item 
If there exists a constant $d$ such that $|\left\langle \phi_{i} , \phi_{j} \right\rangle| =d,$ for $1\leq i< j\leq M,$ then $\{ \phi_{i} \}^{M}_{i=1}$ is an
equiangular frame. 
%\item 
If there exits a constant $c$ such that $\left\|\phi_{i}\right\|_{2} = c$
  for all $i= 1,2,\ldots, n,$ then $\{ \phi_{i} \}^{M}_{i=1}$ is an equal norm frame. If $c=1,$ then it is called a unit norm frame.
%\item 
If a frame is both unit norm and tight, it is called a unit norm tight
 frame (UNTF). 
%\item 
If a frame is both UNTF and equiangular, it is called an equiangular
tight frame (ETF).
The coherence of a frame $\Phi$ is given by $$\mu(\Phi)= \max_{1\leq\; i,j \leq\; M,\; i\neq j} \frac{|\; \phi_i ^T\phi_j|}{\Vert \phi_i\Vert_{2} \Vert \phi_j \Vert_2}.$$  
%There exists a universal lower bound for the coherence of any frame $\Phi$, $\mu(\Phi) \geq \sqrt{\frac{M-m}{m(M-1)}}$, which is known as the Welch Bound. 
Coherence based techniques are used in establishing the guaranteed recovery of sparse signals via Orthogonal Matching Pursuit (OMP) or Basis Pursuit (BP), as summarized by the following result \cite{troop_2010}. 
  \noindent
\begin{theorem}
  An arbitrary $k-$sparse signal $x$ can be uniquely recovered
 % as a solution to problems $P_{0}(\Phi,y),$ 
  using OMP and BP,
  provided 
\begin{equation} \label{eq:OMP_bound}
k < \frac{1}{2}\biggl(1+\frac{1}{\mu(\Phi)}\biggl).
\end{equation}
%\qed
\end{theorem}
% In the next section, we present our main result on the properties of the initial frame that are sufficient to guarantee strict fall in coherence under the action of action of linear invertible operators.
\noindent If $G$ is a nonsingular matrix, then the system $Gy = G\Phi x$ is equivalent to $y=\Phi x$. The bound in (\ref{eq:OMP_bound}) then suggests that if $\mu(G\Phi) < \mu(\Phi)$ both BP and OMP have better performance guarantees when applied on the equivalent system $Gy = G\Phi x$.
%In the next section, we derive the sufficient conditions of an initial frame that can guarantee an equivalent frames with smaller coherence compared to it.
%%%%%%%%%%%%%%%%%%%%%%%%%%%%%%%%%%%%%%%%%%%%%%%%%%%%%%%%%%%%%%%%%%%%%%%%%%%%%%%%%%%%%%%%%%%%%%%%%%%%%%%%%%%%%%%%%%%%%%%%%%%%%%%%%%%%%%%%%%%%%%%%%%%%%%%%%%%%%%%%%%%%%%%%%%%%%%%%%%%%%%%%%%%%%%%%%%%%%%%%%%%%%%%%%%%%%%%%

%%%%%%%%%%%%%%%%%%%%%%%%%%%%%%%%%%%%%%%%%%%%%%%%%%%%%%%%%%%%%%%%%%%%%%%%%%%%%%%%%%%%%%%%%%%%%%%%%%%%%%%%%%%%%%%%%%%%%%%%%%%%%%%%%%%%%%%%%%%%%%%%%%%%%%%%%%%%%%%%%%%%%%%%%%%%%%%%%%%%%%%%%%%%%%%%%%%%%%%%% 
\section{Main results}
\label{main body}
In this section, we present the properties of an initial frame that can ensure strict fall in coherence under the left multiplication of an invertible linear operator. Our main results concerning the sufficient conditions on the initial frames can be given by the following theorem.
%---------------------------
\begin{theorem}
\label{Main Theorem}
For a given unit norm frame $\Phi_{m\times M}$ for $\mathbb{R}^{m}$ with coherence $\mu(\Phi)$, let $\phi_i$ denote the $i^{th}$ column of $\Phi$ and suppose
$$D^{+}_{I_{m}} :=\{(i,j): \phi^{T}_{i}\phi_{j} =\mu(\Phi)\}$$ and, 
$$D^{-}_{I_{m}} :=\{(i,j): \phi^{T}_{i}\phi_{j}=-\mu(\Phi)\}.$$ 
Consider the matrix 
$\Psi_{m^{2}\times (M+|D^{+}_{I_{m}}|+|D^{-}_{I_{m}}|)}=$
%{\tiny{
$$\bigg[
  \bigg(\text{vec}(\phi_{i}\phi^{T}_{i})\bigg)^{M}_{i=1}  \bigg(\text{vec}(\phi'_{ij})\bigg)_{(i,j)\in D^{+}_{I_{m}}}  -\bigg(\text{vec}(\phi'_{ij})\bigg)_{(i,j)\in D^{-}_{I_{m}}}\bigg],$$ 
 % }}
  where
 $\phi'_{ij} :=\frac{\phi_{i}\phi_{j}^{T} + \phi_{j}\phi_{i}^{T}}{2}$, and the $\text{`vec'}$ operation on a matrix of size $m\times M$ produces a vector of length $mM$ by stacking the columns one below the other vertically.  
If there does not exist a vector $r\in \mathbb{R}^{M+|D^{+}_{I_{m}}|+|D^{-}_{I_{m}}|}$ in the nullspace of $\Psi$ satisfying 
$$\sum^{M}_{k=1}r_{k}=-\mu_{\Phi}$$ and
$$\sum^{M+|D^{+}_{I_{m}}|+|D^{-}_{I_{m}}|}_{k=M+1}r_{k}=1,$$ then there exits an invertible operator $G$ such that 
%$\hat{G}\Phi$ is a unit norm frame for $\mathbb{R}^{m}$ and 
$\mu(G\Phi) <\mu(\Phi).$ 
\end{theorem}
\begin{proof}
 Let $S$ be the set of invertible operators $G$ such that $G\Phi$ is a unit norm frame for $\mathbb{R}^{m},$ that is,
$$S=\{G\in \mathbb{R}^{m\times m}: |G|\neq 0, \,  \|G\Phi_{i}\|_{2}=1,\, \forall\, i=1,2,\dots, M\},$$ where $|G|$ denotes the determinant of $G$.
It can be noted that $S \neq \emptyset$ as $I_{m\times m}\in S,$ where $I_{m\times m}$ is the identity matrix. Therefore, in order to show that there exists an invertible operator $G\in S$ such that $\mu_{G\Phi}< \mu_{\Phi},$ it is enough to show that $I_{m\times m}$ is not an optimal solution for the following optimization problem
\begin{equation*}
\begin{aligned}
C_{0}: \; &\underset{G\in S}{\text{arg min}}
&&   \max_{i\neq j}|\langle G\phi_{i},G\phi_{j} \rangle|.\\
%&&& X \succeq 0.
\end{aligned}
\end{equation*}
An equivalent formulation of $C_{0}$ is 
\begin{equation*}
\begin{aligned}
&\underset{X}{\text{arg min}}
& &  \max_{i\neq j}|\phi^{T}_{i} X \phi_{j}|  \\
& \text{subject to}
& &  \phi_{i}^{T} X \phi_{i}= 1, %1,
\quad \forall i=1, \dots , M.\\
&&& X \succ 0,
\end{aligned}
\end{equation*}
where $X= G^{T}G$ and $X\succ 0$ denotes that $X$ is positive definite. The advantage of the equivalent formulation of $C_{0}$ is that the constraints are linear in $X$ and the objective function is convex in $X.$ 
Since the constraint set $$S_{0} = \{X\succ 0 : \phi_{i}^{T} X \phi_{i}= 1,
%\phi_{i}^{T} X \phi_{i}=1, 
\quad \forall i=1, \dots , M \}$$ is convex but not closed, we consider
$$S_{1} = \{X\succeq 0 :  \phi_{i}^{T} X \phi_{i}= 1,
%\phi_{i}^{T} X \phi_{i}=1, 
\quad \forall i=1, \dots , M \},$$ where $X\succeq 0$ implies that $X$ is positive semi-definite with the  corresponding convex optimization problem,
\begin{equation*}
\begin{aligned}
C'_{0}: \; &\underset{X}{\text{arg min}}
& &  \max_{i\neq j}|\phi^{T}_{i} X \phi_{j}|  \\
& \text{subject to}
& &  X \in S_{1},
%, \quad
%\phi_{i}^{T} X \phi_{i}=1,
%\quad \forall i=1, \dots , M.\\
%X \succeq 0,
\end{aligned}
\end{equation*}
%It is easy to see that $S_0$ is contained in  $ S_1.$ Hence, assuming the optimal solution of $C'_{0}$ is positive definite, it is sufficient to show that $I_{m\times m}$ is not the optimal solution of $C'_{0}$ in order to conclude that there exists an invertible operator $\hat{G}\in S$ such that $\mu_{\hat{G}\Phi}< \mu_{\Phi}.$ 
%%%%%%%%%%%%%%%%%%%%%%%%%%%%%%%%%%%%%%%%%%%%%%%%%%%%%%%%%%%%%%%%%%%%%%%%%%%%%%%%%%%

Adding slack and surplus variables $p_{ij}\geq 0$ and $q_{ij}\geq 0$ respectively, one may obtain an equivalent formulation of $C'_{0}$:
\begin{equation*}
  C_1: 
\begin{aligned}
& \underset{X,q,p_{ij},q_{ij}}{\text{max}}
& & (-q) \\
& \text{subject to}
& & \phi_{i}^{T} X \phi_{i}=1,
%\phi^{T}_{i}X\phi_{i} =1, 
\forall i=1,\dots,M ; \\
&&& \phi^{T}_{i}X\phi_{j} + p_{ij} - q =0, \forall 1 \leq i < j \leq M;\\
&&& -\phi^{T}_{i}X\phi_{j} + q_{ij} - q =0, \forall 1 \leq i < j \leq M;\\
&&& X \succeq 0 \\
&&& q \geq 0 \\
&&& p_{ij} \geq 0, \; \forall 1 \leq i < j \leq M;\\
&&& q_{ij} \geq 0, \; \forall 1 \leq i < j \leq M.
\end{aligned}
\end{equation*}
%In $C_{1}$, $X \succeq 0$ means $X$ is a positive semi-positive definite matrix. As before, it may be observed that $q$ acts as the coherence of the matrix $G\Phi$.
%\begin{remark}
 % Let the Cholesky decomposition of $X$ be $L^TL$, then the preconditioner $G$ can be
 % chosen to be $L$. 
%\end{remark}
%\subsection{Formulation as Standard SDP}
%\begin{theorem}
%\label{thm:strictcond}
%Define $\phi'_{ij}=\frac{\phi_{i}\phi_{j}^{T} + \phi_{j}\phi_{i}^{T}}{2}$, a necessary condition for the existence of a pre-conditioner that allows for strict fall in coherence is that there do not exist scalars $z_{ij}$ such that 
%\begin{enumerate}
%\item $\sum^{M}_{i,j=1}z_{ij}\phi'_{ij}=0$
%\item $z_{ij}=0 \quad \mbox{for} \quad (i,j) \notin \{(i,j):|\phi^{T}_{i}\phi_{j}|=\mu_{\Phi}\}$ 
%\item $\sum^{M}_{i,j=1, i<j}(z_{ij}-z_{ji})=1$
%\item $-\sum^{M}_{i=1}z_{ii}=\mu(\Phi).$
%$$\end{enumerate}
%\end{theorem}
%One of the above implication of the above theorem is that for equiangular tight frame $z_{ij}\neq 0$ for all $i\neq j$ therefor the last condition of Theorem~\ref{thm:strictcond}, as a result there can't exists a pre-conditioner such that strict fall in coherence is possible.
Using $M' = \frac{M(M-1)}{2},$ let $\mathbf{0}$ denote the zero matrix of size $m\times m$,
$\mathbf{0'}$ a square zero matrix of size $M'\times M'$, $P$ a
diagonal matrix of size $M'\times M'$ consisting of $p_{ij}$ as diagonal
elements, $Q$ a diagonal matrix of size $M'\times M'$
containing $q_{ij}$ as diagonal elements. Finally, let
$\mathbf{1}_{ij}$ be the diagonal matrix of size $M'\times M'$ whose diagonal
entries are indexed by arranging the tuples $(i,j)$ in lexicographic order
for $1 \leq i < j \leq M$ so that it contains $1$
at the $(i,j)-$th diagonal element and zero elsewhere. For simplicity in
notation, we consider
$\phi'_{ij}=\frac{\phi_{i}\phi_{j}^{T} + \phi_{j}\phi_{i}^{T}}{2}$ and
define the following block matrices:

$F_0 = 
\begin{pmatrix}
\mathbf{0} & 0 & 0 & 0 \\
0 & \mathbf{0'} & 0 & 0 \\
0 & 0 & \mathbf{0'} & 0 \\
0 & 0 & 0 & 1 
\end{pmatrix}, \;
F_{ii}=
\begin{pmatrix}
\phi_{i}\phi_{i}^{T} & 0 & 0 & 0\\
0 & \mathbf{0'} & 0 & 0\\
0 & 0 & \mathbf{0'} & 0 \\
0 & 0 & 0 & 0
\end{pmatrix}$,

$F_{ij}=
\begin{pmatrix}
\phi'_{ij} & 0 & 0 & 0 \\
0 & \mathbf{1}_{ij} & 0 & 0 \\
0 & 0 & \mathbf{0'} & 0 \\
0 & 0 & 0 & -1
\end{pmatrix}, \;
Y =
\begin{pmatrix}
X & 0 & 0 & 0 \\
0 & P & 0 & 0 \\
0 & 0 & Q & 0 \\
0 & 0 & 0 & q
\end{pmatrix}$,

$F_{ji}=
\begin{pmatrix}
-\phi'_{ij}& 0 & 0 & 0\\
0 & \mathbf{0'} & 0 & 0\\
0 & 0 & \mathbf{1}_{ij} & 0\\
0 & 0 & 0 & -1
\end{pmatrix}$.

It is easy to check that, for $1\leq i < j \leq M$, $F_{ii}$, $F_{ij},$
$F_{ji}$ and $F_0$ are symmetric. Using these matrices, we reformulate $C_{1}$ in a standard  Semi-definite Programming (SDP) \cite{Boyd_1996} form  as 
\begin{equation*}
 SDP_{\mathbf{C}_{1}}:
\begin{aligned}
& \underset{Y}{\text{max}}
&& -Tr(F_0Y)\\
& \text{subject to}
&& Tr(F_{ii}Y) = 1,
%1, 
\quad \forall \quad i=1,\dots,M ; \\
&&& Tr(F_{ij}Y) = 0, \quad \forall\quad 1 \leq i < j \leq M;\\
&&& Tr(F_{ji}Y) = 0, \quad \forall \quad 1 \leq i < j \leq M,\\
&&& Y \succeq 0, 
\end{aligned}
\end{equation*}
where $Tr(A)$ represents trace of the matrix $A$ 
and $Y\succeq 0$ implies that $Y$ is positive semi-definite, that is, $\zeta^{T}Y\zeta\geq 0$ for all $\zeta\in \mathbb{R}^{M^{2}-M+m+1}$. 
The dual of $SDP_{\mathbf{C}_{1}}$ is given by
\begin{equation*}
 SDPD_{\mathbf{C}_{1}}:
\begin{aligned}
& \underset{z={\{z_{ij}\}^{M}_{i,j=1}}}{\text{min}}
&& c^{T}z\\
& \text{subject to}
&& F_0 + \sum^{M}_{i=1}z_{ii} F_{ii} +
\sum^{M}_{i,j=1,i<j}z_{ij} F_{ij} + \\
&&& ~~~\sum^{M}_{i,j=1,i<j}z_{ji} F_{ji} \succeq \mathbf{0''},
\end{aligned}
\end{equation*}
where $\mathbf{0''}$ is a square zero matrix of size $M''=M^{2}-M+m+1$ and
$c=\{c_{i}\}^{M^2}_{i=1}\in \mathbb{R}^{M^{2}},$ where $c_{i}$ is 1 for $i=1,2, \dots, M$ and 0 for {\color{black}{$i=M+1, \dots, M^2.$}}
%\sout{the remaining $i$}.
%the remaining  elements are zero.
%{\color{red}{(Pl. check this sentence once .. modified)}}

%%%%%%%%%%%%%%%%%%%%%%%%%%%%%%%%%%%%%%%%%%%%%%%%%%%%%%%%%%%%%%%%%%%%%%%%%%%%%%%%%%%%%%%%%%%%
 It is easy to check that, if $X$ is the identity
matrix, $p_{ij}=1 - \phi_{i}^{T}\phi_{j},$ $q_{ij}=1+
\phi_{i}^{T}\phi_{j}$ and $q=1$, then $Y$ becomes a strict feasible
solution of $SDP_{\mathbf{C}_{1}}$. Similarly, one can verify that
with $z_{ii}=1$ and $z_{ij}= z_{ji} =\frac{1}{M^{2}}$, $Z$ becomes a strict feasible
solution of $SDPD_{\mathbf{C}_{1}}$. Since both primal and dual have strict
feasible solutions, by strong duality (Theorem 3.1 in \cite{Boyd_1996}), optimal
values of primal and dual optimization problems coincide with each other. Consequently, the
duality gap is zero for any optimal pair $(Y^{*},Z^{*}),$ where $Y^{*}$
is an optimal solution for $SDP_{\mathbf{C}_{1}}$ and $Z^{*}$ is an optimal
solution for $SDPD_{\mathbf{C}_{1}}.$ The duality gap being zero implies that 
$0=C^{T}_{M''}Z^{*}+ Tr(F_0Y)=\sum^{M}_{i=1}z^{*}_{ii} + q^{*}$ which implies
further that $q^{*}=-\sum^{M}_{i=1}z^{*}_{ii}.$

The standard optimality condition (Equation (33) in \cite{Boyd_1996}) concerning the primal and dual solutions can be 
written as

$$\begin{aligned}
%& \underset{Y}{\text{max}}
%&& -Tr(F_0Y)\\
%& \text{subject to}
 Tr(F_{ii}Y^{*}) = 1, \quad \forall \quad i=1,\dots,M ; \\
 Tr(F_{ij}Y^{*}) = 0, \quad \forall\quad 1 \leq i < j \leq M;\\
 Tr(F_{ji}Y^{*}) = 0, \quad \forall \quad 1 \leq i < j \leq M,\\
 Y \succeq 0.
\end{aligned}$$
and 
\begin{equation*}
  \begin{split}
    Y^{*}&\left(F_0+\sum^{M}_{i=1}z^{*}_{ii} F_{ii}+\sum^{M}_{i,j=1,i<j}z^{*}_{ij}
      F_{ij}+\sum^{M}_{i,j=1,i<j}z^{*}_{ji}F_{ji}\right) \\
    & \quad =\mathbf{0''},\\
    & (c^{*})^{T}z^{*}=q^{*}.
  \end{split}
\end{equation*}
The above condition results in the following equations: 

\begin{enumerate}[(i)]
\item 
$Tr\bigg(X^{*} (\phi_{i}\phi_{i}^{T})\bigg)=1,  \quad \forall \quad i=1,\dots,M $
\item $Tr(X^{*}\phi'_{ij})+p^{*}_{ij}-q^{*}=0, \quad \forall\quad 1 \leq i < j \leq M$
\item $Tr(-X^{*}\phi'_{ij})+q^{*}_{ij}-q^{*}=0, \quad \forall\quad 1 \leq i < j \leq M$

\item $X^{*}\bigg(\sum^{M}_{i=1}z^{*}_{ii} \phi_{i}\phi^{T}_{i} + \sum^{M}_{i,j=1,i<j}(z^{*}_{ij} - z^{*}_{ji}) \phi'_{ij}\bigg) = 0$
\item $z^{*}_{ij}p^{*}_{ij}=0$
\item $z^{*}_{ji}q^{*}_{ij}=0$
\item $q^{*}\left(1-\sum^{M}_{i,j=1,i<j}(z^{*}_{ij} + z^{*}_{ji})\right)=0$
\item $ - \sum^{M}_{i=1}z^{*}_{ii}=q^{*}$.
\end{enumerate}
If $X^{*}$ is assumed to be positive definite, then the fourth equality above
reduces to \begin{equation}\label{eq:2}
    \sum^{M}_{i=1}z^{*}_{ii} \phi_{i}\phi^{T}_{i} +
\sum^{M}_{i,j=1,i<j}(z^{*}_{ij}-z^{*}_{ji}) \phi'_{ij}=0
\end{equation}

For a positive semi-definite matrix $X=G^{T}G,$ let $D^{+}(X,q)$ and $D^{-}(X,q)$ be the sets of tuples of indices for 
which the corresponding entry of $\Phi^{T}X \Phi$ (equivalently the inner-product between two corresponding columns of 
$G\Phi$) is equal to $q$ and $-q$ respectively, that is, 
$$D^{+}(X,q)=\{(i,j): \phi^{T}_{i}X\phi_{j}=q\}$$
and 
$$D^{-}(X,q)=\{(i,j): \phi^{T}_{i}X\phi_{j}=-q\}.$$
It is clear that $$p^{*}_{ij}=0 \quad \mbox{if and only if } \quad (i,j)\in D^{+}(X^{*},q^{*})$$ and
$$q^{*}_{ij}=0, \quad \mbox{if and only if} \quad (i,j)\in D^{-}(X^{*},q^{*}).$$ From the definition of $D^{+}(X^{*},q^{*})$
and $D^{-}(X^{*},q^{*}),$ it follows that 
\begin{equation}\label{eq:3}
z^{*}_{ij} = 0, \quad \mbox{if} \quad(i,j)\notin D^{+}(X^{*},q^{*})
\end{equation}
and 
\begin{equation} \label{eq:4}
z^{*}_{ji}=0, \quad \mbox{if} \quad(i,j)\notin D^{-}(X^{*},q^{*}).
\end{equation}
Since $q^{*}>0,$
\begin{equation} \label{eq:5}
  \begin{split}
    1 &=\sum^{M}_{i,j=1,i<j}(z^{*}_{ij}+z^{*}_{ji})\\
    &=\sum_{(i,j)\in D^{+}(X^{*},q^{*})}{z^{*}_{ij}}+\sum_{(i,j)\in
      D^{-}(X^{*},q^{*})}{z^{*}_{ji}}.
  \end{split}
\end{equation}

%%%%%%%%%%%%%%%%%%%%%%%%%%%%%%%%%%%%%%%%%%%%%%%%%%%%%%%%%%%%%%%%%%%%%%%%%%%%%%%%%%%%%%%%%%%%
%%%%%%%%%%%%%%%%%%%%%%%%%%%%%%%%%%%%%%%%%%%%%%%%%%%%%%%%%%%%%%%%%%%%%%%%%%%%%%%%%%%%%%%%%%%
Therefore, a positive definite matrix $X^*$ is optimal if and only if
(\ref{eq:2}), (\ref{eq:3}),  (\ref{eq:4}), (\ref{eq:5}) and $- \sum^{M}_{i=1}z^{*}_{ii}=q^{*}$ are satisfied. In 
other words if there exist $z^*$ satisfying the above equations then $X^*$ is optimal. %Now, to show that
Hence, if $X^* = I_{m\times m}$ is not an optimal solution 
%it is sufficient to prove that 
there do not exist scalars  $\{\tilde{z}_{ij}\}^{M}_{i,j=1}$ satisfying the following conditions
  \begin{enumerate}
  \item $\tilde{z}_{ij} = 0$ for  $(i,j)\notin D^{+}(\mathbf{I}_{m},\mu_{\Phi})$,
  \item $\tilde{z}_{ji} = 0$ for  $(i,j)\notin D^{-}(\mathbf{I}_{m},\mu_{\Phi})$,
  \item $\sum^{M}_{i,j=1,i<j}(\tilde{z}_{ij}+\tilde{z}_{ji})=1$,
  \item $\sum^{M}_{i=1}\tilde{z}_{ii} \phi_{i}\phi^{T}_{i} + \sum^{M}_{i,j=1,i<j}(\tilde{z}_{ij}-
  \tilde{z}_{ji})\phi'_{ij}=0.$
  \item $ -\sum^{M}_{i=1}\tilde{z}_{ii}=\mu(\Phi)$
  \end{enumerate}
  
%%%%%%%%%%%%%%%%%%%%%%%%%%%%%%%%%%%%%%%%%%%%%%%%%%%%%%%%%%%%%%%%%%%%%%%%%%%%%%%%%%%%%%%%%%%%%%%
The above conditions can be written as a linear system of equations $\tilde{\Psi} \tilde{z}=y,$ where $\tilde{\Psi}_{(m^{2}+2) \times (M + |D^{+}_{I_{m}}|+ |D^{-}_{I_{m}}|)}=$
{\tiny{
\begin{equation*}
 \begin{bmatrix}
  \bigg(vec(\phi_{i}\phi^{T}_{i})\bigg)^{M}_{i=1} & \bigg(vec(\phi'_{ij})\bigg)_{(i,j)\in D^{+}_{I_{m}}} & -\bigg(vec(\phi'_{ij})\bigg)_{(i,j)\in D^{-}_{I_{m}}}\\
  
  0_{1\times M} & 1_{1\times |D^{+}_{I_{m}}|} & 1_{1\times |D^{-}_{I_{m}}|}\\
  
  1_{1\times M} & 0_{1\times |D^{+}_{I_{m}}|} & 0_{1\times |D^{-}_{I_{m}}|}
  
  \end{bmatrix},
 \end{equation*}}}
 $\tilde{z}_{(M + |D^{+}_{I_{m}}|+ |D^{-}_{I_{m}}|)\times 1}=$
   $$[(\tilde{z}_{ii})^{M}_{i=1} \quad (\tilde{z}_{ij})_{(i,j)\in D^{+}_{I_{m}}} \quad  (\tilde{z}_{ji})_{(i,j)\in D^{-}_{I_{m}}} ]^{T},$$
and   
 $y_{(m^{2}+2)\times 1}=[0_{m^{2}\times 1} \quad 1_{1\times 1} \quad -\mu(\Phi)]^{T}.$
  
Therefore, the infeasiblity of the linear system $\tilde{\Psi}\tilde{z}=y$ guarantees the existence of an operator $\hat{G}$ such that $\hat{G}\Phi$ is a unit norm frame and $\mu(\hat{G}\Phi)<\mu(\Phi).$ 
%This completes the proof.
Despite this, there is no guarantee that 
%the solution of $C_1$ 
$\hat{G}$ is positive definite.
%even if the conditions are satisfied. 
%Let $X^*=(G^*)^TG^*$ be a positive semi-definite solution to $C_1$. 
Nevertheless, there exist
nonsingular matrices $G_n$ so that $G_n \to \hat{G}$ in the Frobenious norm. Since coherence $\mu(G\Phi)$ is a continuous function of $G$ in the Frobenious norm, we have $\mu(G_n\Phi) \to \mu(\hat{G}\Phi)$. This ensures that there exists a non-singular matrix $G$ for which $\mu(G\Phi) < \mu(\Phi).$
%if the conditions of the Corollary are satisfied.
%\end{remark}
\end{proof}

The following corollary follows from the main theorem~\ref{Main Theorem}.
\begin{corollary}
\label{nullspace property}
If $\Psi$ has a trivial nullspace, then there exits an invertible operator $\hat{G}$ such that $\mu(\hat{G}\Phi) <\mu(\Phi).$
\end{corollary}
\begin{proof}
The proof follows from the fact that if $\Psi$ has a trivial nullspace then the two equations given in Theorem~\ref{Main Theorem} can not be satisfied.
\end{proof}

%\subsection{Proof of Corollary~\ref{nullspace property}}
%\label{proof nullspace property}
%\begin{proof}
%The proof follows from the fact that if $\Psi$ has a trivial nullspace then the two equations given in Theorem~\ref{Main Theorem} can not be satisfied.
%\end{proof}

%\subsection{Proof of Corollary~\ref{col square row}}
%\label{proof col square row}
%\begin{proof}
%As $M>m^{2}$ there exists a nonzero vector $z\in \mathbb{R}^{M}$ such that $$\sum^{M}_{i=1}z_{i}\mbox{vec}(\phi_{i}\phi^{T}_{i})=0.$$ Then, $$-\mu_{\Phi}\sum^{M}_{i=1}\frac{z_{i}}{\sum^{M}_{i=1}z_{i}}=-\mu_{\Phi}.$$

%Since $\bigg[
 %  \bigg(\text{vec}(\phi'_{ij})\bigg)_{(i,j)\in D^{+}_{I_{m}}}  \bigg(\text{vec}(\phi'_{ij})\bigg)_{(i,j)\in D^{-}_{I_{m}}}\bigg]$ has a non-trivial nullspace, there exists a non-zero  $h\in \mathbb{R}^{|D^{+}_{I_{m}}|+|D^{-}_{I_{m}}|}$ such that $$\sum^{|D^{+}_{I_{m}}|}_{k=1, (i,j)\in D^{+}_{I_{m}}}h_i \mbox{vec}(\phi_{i}\phi^{T}_{j}) - \sum^{|D^{+}_{I_{m}}|+|D^{-}_{I_{m}}|}_{k=|D^{+}_{I_{m}}| +1, (i,j)\in D^{-}_{I_{m}} }h_k \mbox{vec}(\phi_{i}\phi^{T}_{j}) =0.$$
  % Also, $$\frac{1}{\sum^{|D^{+}_{I_{m}}|+|D^{-}_{I_{m}}|}_{k=1}h_k}\sum^{|D^{+}_{I_{m}}|+|D^{-}_{I_{m}}|}_{k=1}h_k=1.$$
   
%Therefore, $\bigg[\bigg(\frac{-\mu_{\Phi}z_{i}}{\sum^{M}_{i=1}z_{i}}\bigg)^{M}_{i=1} \quad \bigg(\frac{h_k}{\sum^{|D^{+}_{I_{m}}|+|D^{-}_{I_{m}}|}_{k=1}h_k}\bigg)^{|D^{+}_{I_{m}}|+|D^{-}_{I_{m}}|}_{k=1}\bigg]$ belongs to the nullspace of $\Psi$ and satisfies the two linear equations given in Theorem~\ref{Main Theorem}. This completes the proof.

%\end{proof}

So far, we have discussed our theoretical findings. In the next section, we present numerical results in support of our analytical results.

\section{Numerical Observations}
%In this section, we provide numerical evidence to support our theoretical results, presented in section~\ref{main body}. In subsection~\ref{numerical random matrix} 
%and \ref{numerical random binary matrix}, we present the results on the effect of coherence via left action of invertible operators on random Gaussian matrices.
%and random binary matrices respectively. Throughout this numerical study, we  report the average coherence computed over 50 realizations, carrying out simulations on a standard machine with 8 GB RAM using CVX toolbox in MATLAB environment. 

%In this section, we present our empirical results on the effect of coherence via left multiplication of an invertible linear operator. 
%\sout{As we proposed a conditions for strict fall in coherence via preconditioner} 
%We demonstrate the result in theorem \ref{Main Theorem}  %\sout{we result the same here by numerical observations with}
%numerically by considering several matrices. Throughout this numerical study, we  report average coherence computed over 50 realizations, carrying out simulations on a standard machine with 8 GB RAM using CVX in MATLAB environment.  

%\subsection{Effect on coherence of random Gaussian matrices } 
%\label{numerical random matrix}
In this section, we present the effect of coherence on random Gaussian matrices by invertable linear operators $G,$ obtained by solving the $C'_0$ problem (See Section~\ref{main body}).
To begin with, we considered random Gaussian matrices $\Phi \in \mathbb R^{m\times M}$ for different row sizes $m$ along with varying value for the column size $M.$  
%We solve the $C_1$ problem to obtain an invertible  preconditioner $G$, by solving the $C_1$ problem, such that $ \mu(G\phi) < \mu(\phi) $. 
\par As examples, we fixed row sizes as $10$ and $20$, while varying the column sizes, and generated the  Tables~\ref{GPre10} and  \ref{GPre20}.  From Table~\ref{GPre10}, for $M\leq 50,$ it may be noted that  the rank of the corresponding matrix $\Psi$ (as defined in Theorem~\ref{Main Theorem}) and $M+|D^{+}_{I_{m}}|+|D^{-}_{I_{m}}|,$ the column size of $\Psi$, are the same. As a result, $\Psi$ has trivial nullspace. Consequently, from Corollary~\ref{nullspace property}, one can expect a strict fall in coherence. In Table~\ref{GPre10}, we observe the similar behaviour on the coherence as predicted by Corollary~\ref{nullspace property}, that is, for $M\leq 50,$ $\mu(G\Phi)$ is strictly smaller than $\mu(\Phi).$
For $M>50,$ $M+|D^{+}_{I_{m}}|+|D^{-}_{I_{m}}|$ becomes strictly greater than rank of $\Psi$ and we observe from Table~\ref{GPre10} that the coherence remains unchanged.

%From these tables,  it can be seen that for $M=60,70$, the preconditioner does not bring in any fall in coherence. This is because the matrix $A$ (defined in (\ref{eq:A-def})) in such cases has a nonempty null space. That is, the matrix $A$ is not of full rank. This numerical observation is in line with  the Theorem~\ref{Main Theorem}.
%we can observe that, the coherence of $G \Phi$ is less than or equal to $\Phi $. The coherence of $G\phi < \Phi$ up to the matrix of size $10 \times 50 $. For the matrix of sizes $10 \times 60 $ and $10 \times 70 $, there is no reduction in the coherence of precondtioned matrix $G\phi$ against $\phi$. This is because, the null space of the matrix $A$ is not empty i.e., the matrix $A$ is not of full rank. This justifies the results of theorem~\ref{Main Theorem}.    

\begin{table} [h]
%\large
\caption{Solving $C'_0$ for Gaussian random matrices with row size $10$ and column size incremented by $10$ starting with $20.$}\label{GPre10}
\centering
\begin{tabular}[b]{|c|c|c|c|c|}\hline

        $m \times M$ & $\mu(\Phi)$ & $M+|D^{+}_{I_{m}}|+|D^{-}_{I_{m}}|$ & rank $(\Psi)$ & $\mu(G\Phi)$  \\ [0.3ex] 
        \hline
        $10 \times 20$ & 0.7225 & 22 & 22 &  0.4968  \\
        \hline
        $10 \times 30$ & 0.7468 & 32 & 32 &  0.6234  \\
		\hline
        $10 \times 40$ & 0.8738 & 42 & 42 & 0.7541 \\
        \hline
        $10 \times 50$ & 0.8509 & 52 & 52 &  0.8261  \\
        \hline
        $10 \times 60$ & 0.8765 & 62 & 56 &  0.8765  \\
		\hline
        $10 \times 70$ & 0.8626 & 72 & 56 &  0.8626   \\
        \hline
   \end{tabular}
 \end{table}
 
 For a given frame $\Phi$, the authors in \cite{li_2013}, proposed a method to find an invertible operator $G$ 
 for which the associated gram 
 matrix of the transformed frame becomes close to the identity matrix. In other words, the transformed frame 
 becomes close to a unit norm frame and has small coherence. For the frames
 obtained via solving optimization 
 problem described in \cite{li_2013}, we observe similar behaviour on the coherence provided in Table~\ref{GPre20}
 as predicted by Corollary~\ref{nullspace property}. Therefore, we can justify the fall in coherence as described in \cite{li_2013} via proposed null space property.   

%\begin{table} [h]
%\large
%\caption{Gaussian random matrix with row size 10 and column size incremented by 10 starting with 20.}\label{GPre10}
%\centering
%\begin{tabular}[b]{|c|c|c|c|c|c|c|}\hline

 %       $m \times M$ & $\mu(\Phi)$ & car$(A_1)$ & rank $(A_1)$ & $\mu(G\phi)$ & car $(GA_1)$  & rank $(GA_1)$ \\ [0.3ex] 
 %       \hline
 %       $10 \times 20$ & 0.7225 & 22 & 22 &  0.4968 & 92 & 91 \\
 %       \hline
 %       $10 \times 30$ & 0.7468 & 32 & 32 &  0.6234 & 78 & 78 \\
%		\hline
%        $10 \times 40$ & 0.8738 & 42 & 42 & 0.7541 &  72 & 71\\
%        \hline
%        $10 \times 50$ & 0.8509 & 52 & 52 &  0.8261 & 60 & 60 \\
%        \hline
%        $10 \times 60$ & 0.8765 & 62 & 56 &  0.8765 & 62 & 56 \\
%		\hline
 %       $10 \times 70$ & 0.8626 & 72 & 56 & 0.8626 & 72 & 56  \\
%        \hline
%   \end{tabular}
 %\end{table}
 
%\par While in generating  table~\ref{GPre20}, we have fixed the row size of the matrix as 20 by allowing the column size to vary. From this table too, a conclusion sim
%we can observe that, the coherence of $G \Phi$ is less than or equal to $\Phi $. The coherence of $G\phi < \Phi$ up to the matrix of size $20 \times 50 $. For the matrix of sizes $20 \times 60 $ and $20 \times 70 $, there is no reduction in the coherence of precondtioned matrix $G\phi$ against $\phi$. This is because, the null space of the matrix $A$ is not empty. This justifies the results of theorem~\ref{Main Theorem}. 
 
 \begin{table} [h]
%\large
\caption{Applying optimization method in \cite{li_2013} on Gaussian random matrices with row size $300$ and column size incremented by $300$ starting with $610$.}\label{GPre20}
\centering
\begin{tabular}[b]{|c|c|c|c|c|}\hline

        $m \times M$ & $\mu(\Phi)$ & $M+|D^{+}_{I_{m}}|+|D^{-}_{I_{m}}|$ & rank( $\Psi$) & $\mu(G\Phi)$ \\ [0.3ex] 
        \hline
        $300 \times 610 $ & 0.2562 & 612 & 612 &  0.1985  \\
        \hline
        $300 \times 810$ & 0.2742 & 812 & 812 &  0.2219  \\
		\hline
        $300 \times 1010 $ & 0.2898  & 1012 & 1012 & 0.2368  \\
        \hline
        $300 \times 1210 $ & 0.2778 & 1212 & 1212 &  0.2648  \\
        \hline
        $300 \times 1410  $ & 0.2759 & 1412 & 1412 &  0.2536  \\
		\hline
        $300  \times 1510$ & 0.2731 & 1512 & 1512 & 0.2649  \\
        \hline
       
   \end{tabular}
 \end{table}

\section{Concluding remarks}
In the present work, we derived properties of an initial frame that ensure strict fall in coherence via left multiplication by an invertible linear operator. It turns out that the infeasibilty of a linear system of equations obtained from the initial frame results in an equivalent frame with a larger minimum frame angle. In particular, if a certain matrix obtained from initial frame possesses trivial nullspace, then there exists an equivalent frame with strictly smaller coherence.
%Using results from frame theory, we showed that the coherence of binary matrices remains unaltered under the action of invertible diagonal operators.
The numerical results also support our theoretical analysis.


\begin{thebibliography}{}
%\bibitem{donoho_2003}
% Donoho D. L. and Elad M. 2003 \emph{Optimally sparse representation in general (nonorthogonal) dictionaries via l1 minimization}. Proc. Natl. Acad. Sci. USA 100. 5 : 2197 - 2202.
%\bibitem{rao_1997}
%Gorodnitsky, I.F., Rao and B.D., ``Sparse signal reconstruction from limited data using FOCUSS:a re-weighted minimum norm algorithm", IEEE Trans. Signal Process. 45(3), 600-616 (1997).
%\bibitem{tsiligianni_2015}
%E. Tsiligianni, L. P. Kondi and A. K. Katsaggelos, ``Preconditioning for Underdetermined Linear Systems with Sparse Solutions", IEEE Trans. Signal Process.,2015.

%\bibitem{Tropp_2005}
%J. A. Tropp, I. S. Dhillon, R. W. Heath Jr. and T. Strohmer, ``Designing structured tight frames via an alternating projection method," IEEE
%Trans. Inf. Theory, vol. 51, no.1, pp. 188-209, 2005.

\bibitem{alemazkoor_2018}
N. Alemazkoor and H. Meidani, ``A preconditioning approach for improved estimation of sparsepolynomial chaos expansions", Comp. Methods in Appl. Mechanics and Eng.  vol. 342, pp. 474-489, 2018.

%\bibitem{bourgain_2011}
  % J. Bourgainn, S. Dilworth, K. Ford, S. Konyagin and D. Kutzarova,   ``Explicit constructions of RIP matrices and related problems". Duke Math. J., vol. 159, pp. 145-185, 2011.
   


%\bibitem{bruck_2009}
%A. M. Bruckstein,  D.L. Donoho, and M. Elad, ``From sparse solutions of systems of equations to sparse modeling of signals and images", SIAM Rev., vol. 51, pp. 34-81, 2009.



%\bibitem{Gita_2013}
%J. Cahill, P.G. Casazza and G.  Kutyniok, ``Operators and frames", J. of Operator Theory, vol. 70, pp. 145-164, 2013.


%\bibitem{can_2008}
 % E. Candes, ``The restricted isometry property and its implications for compressed sensing", Comptes Rendus Mathematique, vol. 346, pp. 589-592, 2008.
  
 % \bibitem{Candes_2005}
%E. Candes and T. Tao, ``Decoding by linear programming", IEEE Trans. Inf. Theory, vol. 51, pp. 42-4215, 2005


%\bibitem{cas_2012}
% P. G. Casazza and G. Kutyniok, ``Finite frames: Theory and applications", Birkhauser Boston Inc., Boston MA, 2012



\bibitem{chris_2003}
O. Christensen, ``An Introduction to Frames and Riesz Bases", Boston. MA. USA: Birkhauser, 2003.

%\bibitem{daube_1992}
 %I. Daubechies, ``Ten Lectures on Wavelets," Philadelphia, PA, USA: SIAM, 1992.
 
% \bibitem{Ronald_2007}
 % R. A. DeVore, ``Deterministic constructions of compressed sensing matrices", J. of Complexity, vol. 23, pp. 918-925, 2007.
 
 
% \bibitem{Donoho_2006}
%D. Donoho, ``Compressed Sensing", IEEE Trans. Inf. Theory, vol. 52, pp. 1289-1306, 2006.

\bibitem{Elad_2006}
  D.L. Donoho, M. Elad and V.N. Temlyakov, ``Stable Recovery of Sparse Overcomplete Representations in the Presence of Noise", IEEE Trans. Inform. Theory, vol. 52, pp. 6-18, 2006.


\bibitem{duarte_2009}
J. M. Duarte-Carvajalino and G. Sapiro, ``Learning to sense sparse signals: Simultaneous sensing matrix and sparsifying dictionary optimization",  IEEE Trans. on Img. Process., vol. 18, pp. 1395-1408, 2009.




\bibitem{elad_2007}
M. Elad, ``Optimized projections for compressed sensing", IEEE Trans. Signal Process., vol. 55, pp. 5695-5702, 2007.

\bibitem{Elad_2010}
M. Elad, ``Sparse and Redundant Representations; from theory to applications in signal and image processing", Springer, Berlin, 2010.




%\bibitem{eldar_2002}
%Y. C.  Eldar and G. D. Forney, ``Optimal tight frames and quantum measurement", IEEE Trans. Inf. Theory, vol. 48, pp. 599-610, 2002.

\bibitem{mixon_2012}
 M. Fickus, D. G. Mixon, and J. C. Tremain, ``Steiner equiangular tight frames",Linear Algebra Appl., vol. 436, pp. 1014-1027, 2012.






%\bibitem{cand_2006}
%Candes  E. J.  and  Tao T. 2006  \emph{Near-optimal signal recovery from random projections: universal encoding
%strategies}. IEEE Trans. Inf. Theory.  52 : 489-509


%\bibitem{cds_2001}
%Chen S.S. ,  Donoho D.L. and Saunders M.A. 2001 \emph{Atomic Decomposition by Basis Pursuit}. SIAM, Review. 43 : 129-159. 

%\bibitem{holmes_2004}
%R. B. Holmes and V. I. Paulsen,  ``Optimal frames for erasures", Linear Algebra Appl., vol. 377, pp. 31-51, 2004.
 
 %\bibitem{iwen_2014}
% M.A. Iwen, ``Compressed sensing with sparse binary matrices: Instance optimal error guarantees in near-optimal time", Journal of Complexity, vol. 30, no. 1, pp. 1-15, 2014.

   
   
   \bibitem{Kashin_2007}
B.S. Kashin and V.N. Temlyakov, ``A remark on compressed sensing", Matematicheskie Zametki,  vol. 82, pp. 829-837, 2007.

%\bibitem{kova_2007}
%J. Kovacevic and A. Chebira, ``Life beyond bases: The advent of frames (Part I)", IEEE Sig. Process. Mag., vol. 24, pp. 86-104, 2007.


%\bibitem{xu_2010}
%Xu J. , Pi Y.  and Cao Z. 2010  \emph{Optimized projection matrix
%for compressive sensing}. EURASIP J. Adv. Signal Process. 2010 : 560349.

%\bibitem{li_2014}
%S. Li  and G. Ge, ``Deterministic Construction of Sparse Sensing Matrices via Finite Geometry", IEEE Trans. on Sig. Process., vol. 62, pp. 2850-2859, 2014.



\bibitem{li_2013}
G. Li, Z. Zhu, D. Yang, L. Chang  and  H. Bai, ``On Projection Matrix Optimization for Compressive Sensing Systems", IEEE Trans. on Sig. Process., vol. 61, pp. 2887-2898, 2013.




%\bibitem{Tsiligianni_2015}	
%Tsiligianni	E. , Kondi L. P. and  Katsaggelos A. K. 2015 \emph{Preconditioning for Underdetermined Linear Systems with Sparse Solutions}. IEEE Sig. Process. Letters, 22 : 1239-1243.


%\bibitem{ram_2013}
%R. R. Naidu, P. Jampana, C. S. Sastry, ``Deterministic Compressed Sensing Matrices: Construction via Euler Squares and Applications", IEEE Trans. on Sig. Process., vol. 64, pp. 3566-3575, 2016.
	
	




%\bibitem{RM1} 
%S. Pehlivan, D. Han and R. Mohapatra, ``Spectrally two uniform frames for Erasers", Operators and Matrices, vol. 9, pp. 383-399, 2015


\bibitem{UoB}
 P. Sasmal, P.V. Jampana and C. S. Sastry, ``Construction of highly redundant incoherent unit norm tight frames as a union of orthonormal bases", Journal of Complexity, vol. 54, pp. 101401, 2019.




\bibitem{pra_2015}
 P. Sasmal, C. S. Sastry and P. V. Jampana, ``On the existence of equivalence class of RIP-compliant matrices", 2015 International Conference on Sampling Theory and Applications (SampTA), Washington, DC, pp. 274-277, 2015.








	






\bibitem{stro_2003}
T. Strohmer and R. W. Heath, ``Grassmannian frames with applications to coding and communication", Appl. Comput. Harmon. Anal., vol. 14, pp. 257-275, 2003.















%\bibitem{Donoho_2009}
%J. L. Massey and T. Mittelholzer, ``Observed universality of phase transition in high dimensional geometry, with implications for modern data analysis and signal processing," Phil. Trans. R. Soc. A (2009) 367, 4273-4293.



   
   
   

   

%\bibitem{tropp_2004}
 %J. A. Tropp, ``Greed is good: Algorithmic results for sparse approximation", IEEE Trans. Inf. Theory, vol. 50, pp. 2231-2242, 2004.





%\bibitem{Tao_2006}
% E. Candes, J. Romberg and T. Tao, ``Stable signal recovery from incomplete and inaccurate measurements," Comm. Pure and Appl. Math, 59, 1207-1223, 2006.




\bibitem{troop_2010}
J. A. Tropp and  S. J. Wright, ``Computational methods for sparse solution of linear inverse problems", Proceedings of the IEEE, vol. 98, pp. 948-958, 2010.




  
%\bibitem{li_2012}
% Li S. , Gao F. , Ge G. , and Zhang S. 2012 \emph{Deterministic construction of compressed sensing matrices via algebraic curves} IEEE Trans. Inf. Theory. 58 : 5035-5041.
  
%\bibitem{ind_2008}
  %P. Indyk, ``Explicit constructions for compressed sensing matrices," in Proc. 19th Annu. ACM-SIAM Symp. Discr. Algorithms, 30-33, 2008.
  %
%\bibitem{amini_2011}
%Amini  A.  and  Marvasti F. 2011 \emph{Deterministic construction of binary, bipolar and ternary compressed sensing matrices}. IEEE Trans. Inf. Theory. 57 : 2360-2370.







%\bibitem{Grant_2006}
 %Grant M., Boyd S. , and Ye Y. 2006 \emph{Disciplined Convex Programmin}. L. Liberti and N. Maculan, eds., Global Optimization: From Theory to Imp. Springer. 7 : 155-210,.


	
%\bibitem{Helmberg_1996}
%Helmber	C. , Rendl F. , Vanderbei R. J. , and Wolkowicz H. 1996 \emph{An Interior point method for semidefinite programming}. J. Optimization SIAM. 6 : 342-361.



\bibitem{Boyd_1996}
L. Vandenberghe and S. Boyd, ``Semidefinite Programming", SIAM Review, vol. 38, pp. 49-95, 1996.
	



%%%%%%%%%%%%%%%%%%%%%%%%%%%%%%%%%%%%%%%%%%%%%%%%%%%%%5











%\bibitem{kova_2002}
%J. Kovacevic, P. L. Dragotti and V. K. Goyal, ``Filter Bank Frame Expansions With Erasures," IEEE TRANSACTIONS ON INFORMATION THEORY, VOL. 48, NO. 6, JUNE 2002.


%\bibitem{duff_1952}
%R. J. Duffin and A. C. Schaeffer, ``A class of nonharmonic fourier series," Trans. Amer. Math. Soc., vol. 72, no. 2, pp. 341-366, 1952.



%\bibitem{kova2_2007}
%J. Kovacevic and A. Chebira, ``Life beyond bases: The advent of frames(Part II)," IEEE Signal Process. Mag., vol. 24, no. 5, pp. 115-125,
%Sep. 2007.




%\bibitem{waldron_2003}
%S. Waldron, ``Generalized welch bound equality sequences are tight frames," IEEE Trans. Inf. Theory, vol. 49, no. 9, pp. 2307-2309, Sep. 2003.

%\bibitem{griv_2003}
%R. Gribonval and M. Nielsen, ``Sparse representations in unions of bases," IEEE Trans. Inf. Theory, vol. 49, no. 12, pp. 3320-3325, Dec. 2003.



%\bibitem{viswa_1999}
%P. Viswanath and V. Anantharam, ``Optimal sequences and sum capacity of synchronous CDMA systems,"' IEEE Trans. Inf. Theory, vol. 45, no. 6, pp. 1984-1991, Sep. 1999.

%\bibitem{viswa_2002}
%P. Viswanath and V. Anantharam, ``Optimal sequences for CDMA under colored noise: A Schur–Saddle function property," IEEE Trans. Inf. Theory, vol. 48, no. 6, pp. 1295-1318, Jun. 2002.

%\bibitem{goyal_1998}
%V. K. Goyal, M. Vetterli, and N. T. Thao, ``Quantized overcomplete expansions in RN: Analysis, synthesis, and algorithms," IEEE Trans. Inf. Theory, vol. 44, no. 1, pp. 16-31, Jan. 1998.



%\bibitem{fickus_2012}
%M. Fickus and D. G. Mixon, ``Numerically erasure-robust frames," Linear Algebra Appl., vol. 437, no. 6, pp. 1394-1407, 2012.

%\bibitem{elad_2011}
%M. Elad, ``Sparse and Redundant Representations: From Theory to Applications in Signal and Image Processing," New York, NY, USA: Springer-Verlag, 2010.



%\bibitem{donoho_2003}
%D. L. Donoho and M. Elad, ``Optimally sparse representation in general (nonorthogonal) dictionaries via l1-minimization," Proc. Nat. Acad. Sci. USA, vol. 100, no. 5, pp. 2197-2202, 2003.

%\bibitem{troop_2005}
%J. A. Tropp, I. S. Dhillon, R. W. Heath, and T. Strohmer, ``Designing structured tight frames via an alternating projection method," IEEE Trans. Inf. Theory, vol. 51, no. 1, pp. 188-209, Jan. 2005.

%\bibitem{bodmann_2010}
%B. G. Bodmann and P. G. Casazza, ``The road to equal-norm Parseval frames," J. Funct. Anal., vol. 258, no. 2, pp. 397-420, 2010.

%\bibitem{casazza_2012}
%P. G. Casazza, M. Fickus, and D. G. Mixon, ``Auto-tuning unit norm frames," Appl. Comput. Harmon. Anal., vol. 32, no. 1, pp. 1-15, 2012.

%\bibitem{xia_2005}
%P. Xia, S. Zhou, and G. B. Giannakis, ``Achieving the Welch bound with difference sets," IEEE Trans. Inf. Theory, vol. 51, no. 5, pp. 1900-1907,
%May 2005.




%\bibitem{Richard_2008}
  %R. Baraniuk, M. Davenpor, R. De Vore, and M. Wakin, ``A Simple Proof of the Restricted Isometry Property for Random Matrices," Constructive Approximation, 28(3),253-263, 2008.



%\bibitem{zhang_2013}
%Y. Zhang, "`Theory of Compressive Sensing via $l_1$-Minimization: a Non-RIP Analysis and Extensions" , 2013.








% \bibitem{oey_2014}
% E. Oey, ``Projection matrix design for compressive sensing," 2014 Makassar International Conference on Electrical Engineering and Informatics (MICEEI), Makassar, 2014, pp. 124-129.

%\bibitem{troop_2005}
%J. A. Tropp, “On the conditioning of random subdictionaries,” Appl.Comput. Harmon. Anal., vol. 25, no. 1, pp. 1-24, 2008.
%J. A. Tropp, I. S. Dhillon, R. W. Heath, and T. Strohmer, ``Designing structured tight frames via an alternating projection method", IEEE
%Trans. Inf. Theory, vol. 51, no. 1, pp. 188–209, Jan. 2005.


%\bibitem{apostal_1985}
%Apostol and Constantin, ``The reduced minimum modulus", The Michigan Mathematical Journal 32 (1985), no. 3, 279-294, 1985.
%\bibitem{kye_2008}
%Seung-Hyeok Kye, `` The minimum modulus of a linear map in operator spaces", Communications of the Korean Mathematical Society, Vol. 23, No. 4, October 2008.


%\bibitem{bestk_2009}
  % A. Cohen, W. Dahmen and R. DeVore, ``Compressed sensing and best k-term approximation," J. Amer. Math. Soc, Vol. 22, No. 1. (2009).
   





   
%\bibitem{Garnaev_1984}
  % A. Garnaev and E. Gluskin, ``The widths of a Euclidean ball," Dokl. Akad. Nauk USSR 277 (1984),
%1048 1052; English transl. in Soviet Math. Dokl. 30, 200-204, 1984.    
   
%\bibitem{Kashin_1978}
   %B.S. Kashin, ``Widths of certain finite-dimensional sets and classes of smooth functions," Izv. Akad. Nauk SSSR, Ser.Mat. 41 (1977); English transl. in Math. USSR IZV. 11, 317-333, 1978.

%\bibitem{lidl_1997}
   %Lidl  Rudolf and Niederreiter Harald, ``Finite Fields," 2nd ed, Cambridge University Press, New York, 275-277, 1997.

%\bibitem{Lidl_2011}
   %J. L. Nelson and Vladimir N. Temlyakov, ``On the size of incoherent systems," Journal of Approximation Theory, 163(9),1238-1245, 2011.
   

%\bibitem{katona_1964}
   %G. Katona, T. Nemetz and M. Simonovits, ``On a graph-problem of Turan," Mat. Lapok, 15, 228-238, 1964.


  
%\bibitem{gil_2010}
%A. Gilbert and P. Indyk, ``Sparse recovery using sparse matrices," Proc. IEEE, vol. 98, no. 6, pp. 937-947, 2010.

%\bibitem{Bruckstein_2009}	
	%A. M. Bruckstein, D. L. Donoho, M. Elad, ``From sparse solutions of systems of equations to sparse modeling of signals and images," SIAM Review, Vol. 51, No. 1, pp: 34-81,  2009.
%\bibitem{rusu_2015}
%C Rusu and N. G. Prelcic ``Designing Incoherent Frames Through Convex Techniques for Optimized Compressed Sensing," IEEE Transactions on Signal Processing 64(9):1-1, May 2016.

%\bibitem{Rusu_2013}	
	%C Rusu, ``Design of Incoherent Frames via Convex Optimization," IEEE Signal Processing Letters, Vol. 20, no. 7, 2013.

%\bibitem{tsiligianni_2014}
%E. Tsiligianni, L. P. Kondi, A. K. Katsaggelos, ``Construction of Incoherent Unit Norm Tight Frames With Application to Compressed Sensing," IEEE Transactions on Information Theory, VOL. 60, NO. 4, 2014
	





	


%\bibitem{dima_2012}
 %A. G. Dimakis, R. Smarandache, and P. O. Vontobel,``LDPC codes for compressed sensing," IEEE Trans. Inf. Theory, vol. 58, no. 5, pp. 3093-3114, May 2012.
%\bibitem{mixon_2014}
 %Dustin G. Mixon,``Explicit Matrices with the Restricted Isometry Property: Breaking the Square-Root Bottleneck" arXiv:1403.3427v1, 2014.
%
%\bibitem{adcock_2013}
 %B. Adcock,A. Hansen,C. Poon,B. Roman, ``Breaking the coherence barrier: A new theory for compressed sensing" arXiv:1302.0561v3, 2013.
%%
%\bibitem{bryant_2014}
%D. Bryant, P. Cathain, ``An asymptotic existence result on compressed sensing matrices" ,arXiv:1403.2807v1, 2014.
%
%\bibitem{mixon_2012}
%M. Fickus, D. G. Mixon, and J. C. Tremain, ``Steiner equiangular tight frames", Linear Algebra Appl., 436(5):1014{1027, 2012.
%\bibitem{rauhut_2010}
%S. Foucart and H. Rauhut, ``A Mathematical Introduction to Compressed Sensing", Applied and Numerical Harmonic Analysis, 2010.


%\bibitem{apple_2009}
 %L. Applebauma, S. D. Howardb, S. Searlec, and R. Calderbank, ``Chirp sensing codes: Deterministic compressed sensing measurements for fast recovery," Appl.
%Comput. Harmon. Anal., 283-290, Mar. 2009.







 
 \bibitem{zhang_2013}
Y. Zhang, ``Theory of compressive sensing via $\ell_1-$minimization: A non-rip analysis and extensions", Journal of the Operations Research Society of China, vol. 1, pp. 79-105, 2013.



\end{thebibliography}
\end{document}